\documentclass[12pt]{article}
\usepackage{amsmath, amssymb, amscd, amsthm, amsfonts, xfrac, mathrsfs,diagbox,float}
\numberwithin{equation}{section}
\usepackage{graphicx}
\usepackage{hyperref}
\usepackage{authblk}
\usepackage{url}
\hypersetup{backref, colorlinks=true}
\newtheorem{theorem}{Theorem}[section]

\newtheorem{lemma}[theorem]{Lemma}

\title{\textbf{Inequalities related to the coefficients of the $j$-function}}
\author{Zhongjie Li}
\affil{School of Mathematics and KL-AAGDM \break Tianjin University \break Tianjin 300350, China \break\texttt{lizhongjie@tju.edu.cn}}
\date{}

\begin{document}
\maketitle
\begin{abstract}
In recent years, the log-concavity or log-convexity of combinatorial sequences and their root sequences, higher order Tur{\'a}n inequalities, and Laguerre inequalities of order two have been widely studied. However, the research of the Fourier coefficient $c(n)$ of the $j$-function is limited to its asymptotic form. In this paper, we give the appropriate upper and lower bounds of $c(n)$ to establish the inequalities associated with it.
\end{abstract}

\noindent{\textbf{Keywords}:}
$j$-function, log-convexity, log-concavity, higher order Tur\'{a}n inequality, Laguerre inequality of order two.

\section{Introduction}
Let \textbf{H} denote the upper half-plane $\{z\in\mathbb{C}:\Im(z)>0\}$. The modular invari- ant $j$ is defined by
\begin{align*}
j(\tau) & = \frac{(1 + 240 \sum_{n \ge 1} \sigma_{3}(n) q^{n})^3}{q \prod_{n\ge 1} (1-q^{n})^{24}}\\
& = q^{-1} + 744 + 196884q + 21493760q^{2} + \cdots
\end{align*}
\begin{align*}
= q^{-1} + \sum_{n\ge 0} c(n) q^{n},
\end{align*}
where $\tau \in $ \textbf{H}, $\sigma_{3}(n) = \sum_{d|n} d^3$ and $q = e^{2\pi i \tau}$.\\
\indent In the past few decades, the coefficient $c(n)$ of the $j$-function has been extensively studied. Petersson \cite{petersson1932uber} and Rademacher \cite{rademacher1938fourier} independently used the circle method introduced by Hardy and Ramanujan \cite{hardy1918asymptotic} to derive the asymptotic formula,
\[c(n)\sim \frac{e^{4 \pi \sqrt{n}}}{\sqrt{2} n^{3/4}},\ n \rightarrow \infty. \]
Subsequently, Dewar and Murty \cite{dewar2013asymptotic} obtained the same result using Laplace's  method of steepest descent.\\
\indent To establish more effective estimates, Mahler \cite{mahler1974coefficients} proved that
\[c_{m}(n) = [q^{n}]j^{m} \le 1200 e^{4 \pi \sqrt{m(m+n)}},\quad \forall\ n, m \ge 1,\] 
while Herrmann \cite{herrmann1975uber} showed that 
\[c(n) \le 6 e^{4 \pi \sqrt{n}}, \quad \forall\ n \ge 1.\] In 2005, Brisebarre and Philibert \cite{brisebarre2005effective} derived precise lower and upper bounds for $c(n)$, which are utilized in this article. \\
\indent In this article, we primarily focus on inequalities associated with $c(n)$, including log-concavity, log-convexity, higher-order Tur{\' a}n inequality and  Laguerre inequality of order 2. Recall that a sequence $\{a_{n}\}_{n \ge 0}$ is called \textit{log-convex} ({\it log-concave}, respectively) if for all $n \ge 0$,
\begin{align}\label{1.1}
a_{n}a_{n+2} - a_{n+1}^{2} \ge 0\ (\le 0,\ respectively).
\end{align}
A sequence $\{a_n\}_{n \ge 0}$ is \textit{ratio log-convex} ({\it ratio log-concave}, respectively) if the sequence $\{\frac{a_{n+1}}{a_{n}}\}_{n \ge 0}$ is log-convex (log-concave, respectively). For a non-negative sequence $\{a_n\}_{n \ge 1}$, its {\it root sequence} is defined as $\{\sqrt[n]{a_n}\}_{n \ge 1}$. \\
\indent Liu and Wang \cite{liu2007log} gave criteria for three-term recursive sequences to satisfy log-convexity. Xia \cite{xia2018log} developed a systematic method to prove the log-concavity of $\{\sqrt[n]{a_n}\}_{n \ge 1}$ when $a_n$ satisfies a second-order linear recurrence. Numerous combinatorial sequences have been shown to be log-convex, and the root sequences of combinatorial sequences have been confirmed to be log-concave, see \cite{sun2020log, xia2025log, zhao2015log, zhao2018log} for further details.\\
\indent A sequence $\{a_n\}_{n \ge 1}$ is said to satisfy the \emph{higher order Tur{\'a}n inequalities} or the \emph{cubic Newton inequalities} if for all $n \ge 2$,
\begin{align}\label{1.2}
4(a_{n}^2 - a_{n-1}a_{n+1})(a_{n+1}^2 - a_{n}a_{n+2}) - (a_{n}a_{n+1} - a_{n-1}a_{n+2})^2 \ge 0.
\end{align}
The Tur{\'a}n inequality and the higher order Tur{\'a}n inequality are related to the Laguerre-P\'{o}lya class of real entire functions. A real entire function
\begin{align*}
\psi(x) = \sum_{n=0}^{\infty} a_{n} \frac{x^n}{n!}
\end{align*}
is said to be in the Laguerre-P\'{o}lya class, denoted $\psi(x) \in \mathcal{LP}$, if 
\begin{align*}
\psi(x) = cx^{m}e^{-\alpha x^2 + \beta x}\prod \limits_{k=1}^{\infty}(1 + x/x_{k})e^{-x/x_{k}},
\end{align*}
where $c,\ \beta,\ x_{k}$ are real numbers, $\alpha \ge 0$, $m$ is a nonnegative integer and $\sum x_{k}^{-2} < \infty$. Dimitrov \cite{dimitrov1998higher} proved that if a real entire function $\psi(x) \in \mathcal{LP}$, its Maclaurin coefficients satisfy the higher Tur{\' a}n inequality (\ref{1.2}).\\
\indent Hou and Li \cite{hou2021log} proposed an approach to analyze the higher order Tur{\' a}n inequalities for the P-recursive sequence $\{a_n\}_{n\ge N}$. Wang \cite{wang2019higher} presented a unified method to studying the higher order Tur{\'a}n inequalities for the sequence $\{\frac{a_{n}}{n!}\}_{n \ge 0}$ when $a_n$ satisfies a second-order linear recurrence. Many combinatorial sequences have been demonstrated
to satisfy the higher Tur{\'a}n inequality (\ref{1.2}), see \cite{chen2019higher, griffin2019jensen, liu2021inequalities} for more details.\\
\indent A polynomial $f(x)$ satisfies the Laguerre inequality if
\[f^{'}(x)^{2} - f(x) f^{''}(x) \ge 0.\]
In 1913, Jensen \cite{jensen1913recherches} introduced the $n$-th generalization of the Laguerre inequality as
\begin{align}\label{1.3}
L_n(f(x)):=\frac{1}{2}\sum_{k=0}^{2n}(-1)^{n+k}\binom{2n}{k}f^{(k)}(x)f^{(2n-k)}(x)\geq0,
\end{align}
where $f^{(k)}(x)$ denotes the $k$th derivative of $f(x)$.\\
Similarly, Wang and Yang \cite{wang2022laguerre} defined that a sequence $\{a_n\}_{n\ge 1}$ satisfies the Laguerre inequality of order $m$ if for $n\ge 1$,
\begin{align}\label{1.4}
L_m(a_n):=\frac{1}{2}\sum_{k=0}^{2m}(-1)^{k+m}\binom{2m}{k}a_{n+k}a_{2m-k+n}\geq0.
\end{align}
\indent Li \cite{li2022log} proposed a method to determine an explicit integer $N$ such that the Laguerre inequality of order 2 holds for a P-recursive sequence $\{a_n\}_{n\ge N}$. Wang and Yang \cite{wang2022laguerre} proved that the partition function, the
overpartition function, the Bernoulli numbers, the derangement numbers, the Motzkin numbers, the Fine numbers, the Franel numbers and the Domb numbers all satisfy the Laguerre inequalities of order two.\\
\indent In this paper, we first concretize the upper and lower bounds of $c(n)$ given in \cite{brisebarre2005effective} to get the upper and lower bounds we need in this article. Subsequently, we get the upper and lower bounds of $r(n) = \frac{c(n+1)}{c(n)}$ and $u(n) = \frac{c(n-1)c(n+1)}{c(n)^2}$.  By integrating the sufficient conditions proposed in this paper for the root sequences to satisfy log-concavity or log-convexity, the sufficient conditions for higher order Tur{\' a}n inequalities from \cite{hou2021log} and the sufficient conditions for second-order Laguerre inequalities from \cite{li2022log}, we systematically derive the inequalities related to $c(n)$.\\
\indent This paper is organized as follows. In Section \ref{s2}, we establish the lower and upper bounds for $c(n)$, $r(n)$ and $u(n)$. In Section \ref{s3}, we prove the sequence $\{c(n)\}_{n \ge 0}$ and the root sequence $\{\sqrt[n]{c(n)}\}_{n \ge 1}$ satisfy log-concavity or log-convexity. In Section \ref{s4}, we demonstrate that $c(n)$ satisfies both the higher order Tur{\'a}n inequality and the second-order Laguerre inequality. 

\section{The upper and lower bounds related to c(n)}\label{s2}
In this section, we will establish the upper and lower bounds for $c(n)$, $r(n)$, and $u(n)$, which are essential for Section \ref{s3} and Section \ref{s4}. In 2005, Brisebarre and Philibert \cite{brisebarre2005effective} derived effective bounds for the Fourier coefficients $c(n)$ of the modular invariant $j$.  
\begin{lemma}\rm{\bf{(\cite[Theorem 1.3]{brisebarre2005effective})}}\label{l2.1}
For $k \ge 0$, let
\[
(1,k) = \frac{\prod \limits_{j=0}^{k-1}(4 - (2j+1)^{2})}{4^{k}k!}.
\]
For all $n \ge 1$, $p \ge 1$, we have
\[
c(n) = \frac{e^{4 \pi \sqrt{n}}}{\sqrt{2} n^{3/4}}\left(\sum_{k = 0}^{p-1} \left(-\frac{1}{8 \pi}\right)^{k}\frac{(1,k)}{n^{k/2}} + \frac{r_{p}(n)}{n^{p/2}}\right),
\]
where
\[
|r_{p}(n)| \le \frac{1}{\sqrt{2}}\frac{|(1,p)|}{(4 \pi)^{p}} + 62\sqrt{2} e^{-2 \pi \sqrt{n}} n^{p/2}.
\]
\end{lemma}
\indent By setting $p = 5$ in the lemma above, and verifying the values of the first 15 terms of $c(n)$, we obtain the following upper bound $b(n)$ and lower bound $a(n)$.
\begin{lemma}\label{l2.2}
For $n \ge 6$, we have
\begin{align*}
a(n) & = \frac{e^{4\pi \sqrt{n}}}{\sqrt{2} n^{3/4}}\left(d(n) - \frac{2.21073\cdot 10^{-5}}{n^{5/2}}\right),\\
b(n) & = \frac{e^{4\pi \sqrt{n}}}{\sqrt{2} n^{3/4}}\left(d(n) + \frac{2.21073\cdot 10^{-5}}{n^{5/2}}\right),
\end{align*}
where
\begin{align*}
d(n) = 1 - \frac{3}{32\pi n^{1/2}}- \frac{15}{2048 \pi ^{2} n} - \frac{105}{65536 \pi ^{3}n^{3/2}} - \frac{4725}{8388608 \pi^{4} n^{2}}.
\end{align*}
\end{lemma}
\indent In order to keep the proof process simple, we occasionally employ the following simplified lower bound $e(n)$ of $c(n)$ in subsequent derivations.
\begin{theorem}\label{t2.3}
For $n \ge 2$, we have
\begin{align}\label{2.1}
e(n) = \frac{e^{4\pi \sqrt{n}}}{\sqrt{2}n}.
\end{align}
\end{theorem}
\begin{proof}
To prove that $e(n) \le c(n)$, it suffices to show that $e(n) \le a(n)$, so we only need to prove that the following inequality holds,
\begin{align*}
A(n) = n^{1/4}\left(d(n) - \frac{2.21073\cdot 10^{-5}}{n^{5/2}}\right) - 1 \ge 0.
\end{align*}
By calculation, we know that $A(n) = 0$ has a unique real root $a_{0} \in (1,2)$. Since $A(1) < 0$ and $A(2) > 0$, it follows that $A(n) > 0$ for all $n \ge 2$.

\end{proof}
\indent Using the inequality $a(n) \le c(n) \le b(n)$, we can deduce that $\frac{a(n+1)}{b(n)} \le r(n) \le \frac{b(n+1)}{a(n)}$. Subsequently, by comparing the first six terms of $c(n)$, we get the upper bound $g(n)$ and lower bound $f(n)$ for $r(n)$.
\begin{theorem}\label{t2.4}
For $n \ge 1$, we have
\begin{align}
f(n) & = h(n) + \frac{25.9592}{n^{5/2}} - \frac{1.78935}{n^{3}} - \frac{18.3114}{n^{7/2}} - \frac{11.5407}{n^4} ,\label{2.2}\\
g(n) & = h(n) + \frac{25.9593}{n^{5/2}} - \frac{0.78879}{n^{3}} ,\label{2.3}
\end{align}
where
\begin{align*}
h(n) = 1 + \frac{2 \pi}{n^{1/2}} + \frac{-\frac{3}{4} + 2 \pi^{2}}{n} + \frac{\frac{3}{64\pi} - 2\pi + \frac{4\pi^{3}}{3}}{n^{3/2}} + \frac{\frac{3}{4}+\frac{3}{256\pi^{2}} - \frac{5\pi^{2}}{2}+\frac{2\pi^{4}}{3}}{n^{2}}.
\end{align*}
\end{theorem}
\indent In Section \ref{s4}, we need to use the bounds of $u(n)$, which imposes higher precision requirements on the bounds of $c(n)$. Therefore, by taking $p=6$ and verifying the values of the first 15 terms of $c(n)$, we derive the following more precise lower bound $k(n)$ and upper bound $l(n)$.
\begin{lemma}\label{l2.5}
For $n \ge 7$, we have
\begin{align*}
k(n) & = \frac{e^{4\pi \sqrt{n}}}{\sqrt{2} n^{3/4}}\left(m(n) - \frac{1.57696\cdot 10^{-5}}{n^{3}}\right),\\
l(n) & = \frac{e^{4\pi \sqrt{n}}}{\sqrt{2} n^{3/4}}\left(m(n) + \frac{1.57696\cdot 10^{-5}}{n^{3}}\right),
\end{align*}
where
\begin{align*}
& m(n) \\
={} & 1 - \frac{3}{32\pi n^{\frac{1}{2}}}- \frac{15}{2048 \pi ^{2} n} - \frac{105}{65536 \pi ^{3}n^{\frac{3}{2}}} - \frac{4725}{8388608 \pi^{4} n^{2}} - \frac{72765}{268435456 \pi^{5} n^{\frac{5}{2}}}.
\end{align*}
\end{lemma}
\indent Using inequality $k(n) < c(n) < l(n)$, we can derive $\frac{k(n-1)k(n+1)}{l(n)^{2}} < u(n) < \frac{l(n-1)l(n+1)}{k(n)^{2}}$. At the same time, by comparing the first 7 terms of $c(n)$, and finally we obtain the lower bound $p(n)$ and upper bound $q(n)$ for $u(n)$.
\begin{theorem}\label{t2.6}
For $n\ge 1$, we have
\begin{align}
p(n) &= 1 - \frac{\pi}{n^{3/2}} + \frac{\frac{3}{4}}{n^{2}} - \frac{\frac{9}{128\pi}-1}{n^{5/2}} ,\label{2.4}\\
q(n) &= 1 - \frac{\pi}{n^{3/2}} + \frac{\frac{3}{4}}{n^{2}} - \frac{\frac{9}{128\pi}}{n^{5/2}} + \frac{5.93249}{n^{3}}.\label{2.5}
\end{align}
\end{theorem}

\section{Log-concavity and log-convexity}\label{s3}
In this section, we first apply the bounds (\ref{2.2}) and (\ref{2.3}) of $r(n) = \frac{c(n+1)}{c(n)}$ to derive both the log-concavity and ratio log-convexity of $c(n)$. 
\begin{theorem}\label{t3.1}
The sequence $\{c(n)\}_{n \ge 0}$ is log-concave and the sequence $\{c(n)\}_{n \ge 1}$ is ratio log-convex.
\end{theorem}
\begin{proof}
To establish the log-convexity of $c(n)$, it suffices to demonstrate that the following inequality holds,
\[\frac{c(n+1)}{c(n)} \ge \frac{c(n+2)}{c(n+1)}.\]
Thus, we only need to verify that 
\[f(n) \ge g(n+1).\]
\indent Through {\tt Mathematica}, we can conclude that $f(n) - g(n+1) = 0$ has a unique real root $b_{0} \in (0,1)$. Given that $f(1) - g(2) >0$, it follows that $f(n) > g(n+1)$ for $n \ge 1$. By verifying the initial values, we confirm that the sequence $\{c(n)\}_{n \ge 0}$ is log-concave.\\
\indent Similarly, by applying the same method, we can prove that the sequence $\{c(n)\}_{n \ge 1}$ is ratio log-convex.
\end{proof}
\indent Before proving that the root sequence $\{\sqrt[n]{c(n)}\}_{n \ge 1}$ is log-convex, we first give a sufficient condition for the log-convexity of the root sequence of a sequence. The proof process is similar to that of \cite[Theorem 3.1]{hou2023log}, thus is  omited here.
\begin{lemma}\label{l3.2}
Let $\{a_n\}_{n \ge 1}$ be a positive sequence. Suppose we can find a lower bound $h_n$ of $a_n$, a lower bound $f_n$ and an upper bound $g_n$ of $\frac{a_{n+1}}{a_n}$, such that
\begin{equation}\label{3.1}
2 \log h_n + n(n+1) \log f_{n+1} - n(n+3) \log g_n \ge 0, \quad \forall\ n \ge N.
\end{equation}
Then the root sequence $\{\sqrt[n]{a_n}\}_{n\ge{N}}$ is log-convex.
\end{lemma}
\indent Next, we use the above lemma and the bounds (\ref{2.1})--(\ref{2.3}) to give the proof process of the log-convexity for the sequence $\{\sqrt[n]{c(n)}\}_{n \ge 1}$.
\begin{theorem}\label{t3.3}
The sequence $\{\sqrt[n]{c(n)}\}_{n \ge 1}$ is log-convex.
\end{theorem}
\begin{proof}
Combined with the above lemma, we only need to prove that the following inequality holds for $n \ge 2$,
\begin{align*}
D(n) = 2 \log e_n + n(n+1) \log f_{n+1} - n(n+3) \log g_n \ge 0.
\end{align*}
\indent By {\tt Mathematica}, we know that
\[
\lim_{n \to +\infty} D(n) = + \infty, \quad \lim_{n \to +\infty} D'(n) = \lim_{n \to +\infty} D''(n) = 0,
\]
and $D^{'''}(n) > 0$ for $n \ge 2$. Consequently, we deduce that $D^{''}(n) < 0$ and $D^{'}(n) > 0$ for all $n \ge 2$. Since $D(2) > 0$, it follows that that $D(n) > 0$ for all $n \ge 2$. By examining the initial values, we finally conclude that the sequence $\{\sqrt[n]{c(n)}\}_{n \ge 1}$ is log-convex.
\end{proof}
\indent Similar to Lemma \ref{l3.2}, we first establish the criterion for the ratio log-concavity of the root sequence of a sequence. We likewise omit the proof process.
\begin{lemma}\label{l3.4}
Let $\{a_n\}_{n \ge 1}$ be a positive sequence. Suppose we can find a lower bound $h_n$ of $a_n$, a lower bound $f_n$ and an upper bound $g_n$ of $\frac{a_{n+1}}{a_n}$, such that for $n \ge N$,
\[
6\log h_n +(n^2 - n)(2n + 5)\log f_n - (n^2 + n)(n + 2)\log g_{n-1} - (n^3 - n)\log g_{n+1} \ge 0.
\]
Then the root sequence $\{\sqrt[n]{a_n}\}_{n \ge N}$ is ratio log-concave.
\end{lemma}
Subsequently, we use Lemma \ref{l3.4}, lower and upper bounds (\ref{2.1})--(\ref{2.3}) to prove the ratio log-concavity of the root sequence $\{\sqrt[n]{c(n)}\}_{n \ge 1}$.
\begin{theorem}\label{t3.5}
The sequence $\{\sqrt[n]{c(n)}\}_{n\ge 1}$ is ratio log-concave.
\end{theorem}
\begin{proof}
According to Lemma \ref{l3.4}, it suffices to prove that when $n \ge 2$, the following inequality holds,
\begin{align*}
& F(n)\\
= {} & 6\log e_n +(n^2 - n)(2n + 5)\log f_n - (n^2 + n)(n + 2)\log g_{n-1} - (n^3 - n)\log g_{n+1}\\
\ge {} & 0.
\end{align*}
\indent By calculation, we can derive that
\[
\lim_{n \to +\infty} D(n) = + \infty, \quad \lim_{n \to +\infty} D^{'}(n) = \lim_{n \to +\infty} D^{''}(n) = \lim_{n \to +\infty} D^{'''}(n) = 0,
\]
and $D^{(4)}(n) < 0$ for $n \ge 2$. By derivative analysis, we deduce that $D^{'''}(n) > 0$, $D^{''}(n) < 0$ and $D^{'}(n) > 0$ hold for $n \ge 2$. Given that $D(2) > 0 $, we obtain that $D(n) > 0$ for all $n \ge 2$. By checking the initial value, we finally conclude that the sequence $\{\sqrt[n]{c(n)}\}_{n \ge 1}$ is ratio log-concave.
\end{proof}

\section{Higher Tur{\' a}n inequality and Laguerre inequality of order 2}\label{s4}
In this section, we first give a sufficient condition for the higher order Tur{\' a}n inequality, which was proposed by Hou and Li \cite{hou2021log}.
\begin{lemma}\label{l4.1}
Let
\[t(x,y) = 4(1-x)(1-y) - (1-xy)^{2} \ \ and \ \ u_n = \frac{a_{n-1}a_{n+1}}{a_{n}^{2}}.\]
If there exist an integer $N$, a lower bound $f_n$ and an upper bound $g_n$ of $u_n$ such that for all $n \ge N$,
\[
f_{n} < u_{n} < g_{n}, 
\]
and\\
\[ t(f_{n}, f_{n+1}) > 0, \quad t(f_{n}, g_{n+1}) > 0, \quad t(g_{n}, f_{n+1}) > 0, \quad t(g_{n}, g_{n+1}) > 0.\]
Then $\{a_n\}_{n \ge N}$ satisfies the higher order Tur{\' a}n inequality.
\end{lemma}

\indent Next, we use the above lemma and bounds (\ref{2.4}) and (\ref{2.5}) to give a proof that the sequence $\{c(n)\}_{n\ge 1}$ satisfies the higher order Tur{\'a}n inequality.
\begin{theorem}\label{t4.2}
The sequence $\{c(n)\}_{n \ge 1}$ satisfies the higher order Tur{\'a}n inequality.
\end{theorem}
\begin{proof}
According to Theorem \ref{t2.6}, $u(n)$ has a lower bound $p(n)$ and an upper bound $q(n)$. Then by Lemma \ref{l4.1}, we only need to prove that the following four inequalities holds for all $n \ge 3$,
\begin{multline*}
t(p({n}), p({n+1})) > 0, \ t(p({n}), q({n+1})) > 0, \  t(q({n}), p({n+1})) > 0, \\  t(q({n}), q({n+1})) > 0.
\end{multline*}
\indent By {\tt Mathematica}, we analyze $t(p(n),\ p(n+1)) = 0$, and find a unique real root $c_0\in (0,1)$. Since $t(p(1),p(2)) > 0$, we deduce that $t(p(n),\ p(n+1)) > 0$ for $n \ge 1$. \\
\indent For $t(p(n),\ q(n+1)) = 0$, there exists a unique real root $d_0\in (1,2)$. Given that $t(p(2),q(3)) > 0$, we can obtain that $t(p(n),\ q(n+1)) > 0$ for $n \ge 2$.\\
\indent Analyzing $t(q(n),\ p(n+1)) = 0$, we derive that there is a unique real root $e_0\in (2,3)$. Since $t(q(3),p(4)) > 0$, it follows that $t(q(n),\ p(n+1)) > 0$ for $n \ge 3$.\\\indent Finally, for $t(q(n),\ q(n+1)) = 0$, there exists a unique real root $f_0\in (1,2)$. Given that $t(q(2),q(3)) > 0$, we conclude that $t(q(n),\ q(n+1)) > 0$ for $n \ge 2$.\\
\indent Thus we can derive that when $n\ge 3$, all four of the above inequalities hold. By examining the initial values, we confirm that the sequence $\{c(n)\}_{n \ge 1}$ satisfies the higher order Tur{\'a}n inequality.

\end{proof}
\indent Next, we present a criterion for determining whether a sequence satisfies Laguerre inequality of order two, which was proposed by Li in \cite{li2022log}.
\begin{lemma}\label{l4.3}
Let
\[u_n = \frac{a_{n-1}a_{n+1}}{a_{n}^{2}}.\]
If there exist an integer $N_{1}$, a lower bound $f_n$ and an upper bound $g_{n}$ of $u_{n}$ such that for all $n \ge N_{1}$,
\[0 < f_n < u_n < g_n,\]
and
\[f_{n-1}f_{n}^{2}f_{n+1} - 4g_{n} + 3 > 0,\quad n\ge N_{2}.\]
Let $N= \max\{N_{1}, N_{2}\}$, then the sequence $\{a_n\}_{n \ge N}$ satisfies Laguerre inequality of order two.
\end{lemma}
\indent Then, we use Lemma \ref{l4.3}, the lower bound $p(n)$, and the upper bound $q(n)$ of $u(n)$ to prove that sequence $\{c(n)\}_{n \ge 1}$ satisfies the Laguerre inequality of order two.
\begin{theorem}\label{t4.4}
The sequence $\{c(n)\}_{n \ge 1}$ satisfies the Laguerre inequality of order two.
\end{theorem}
\begin{proof}
From the Lemma \ref{l4.3}, it suffices to prove that the following inequalities hold for all $n \ge 3$,
\[F(n) = p(n-1)p(n)^{2}p(n+1) - 4q(n) +3 > 0.\]
\indent Using {\tt Mathematica}, we find that $F(n)=0$ has a unique root $f_{0}\in (2,3)$. Since $F(3) > 0$, we deduce that $F(n) > 0$ for $n \ge 3$. By checking the initial values, we finally conclude that the sequence $\{c(n)\}_{n \ge 1}$ satisfies  Laguerre inequality of order 2.
\end{proof}

\section*{Acknowledgement}
This work was supported by the by National Key Research and Development Program of China (foundation number 2023YFA1009401).


\begin{thebibliography}{10}

\bibitem{brisebarre2005effective}
N. Brisebarre and G. Philibert, Effective lower and upper bounds for the Fourier coefficients of powers of the modular invariant $j$, J. Ramanujan Math. Soc. \textbf{20} (2005), 255--282.

\bibitem{chen2019higher}
W. Y. C. Chen, D. X. Q. Jia and L. X. W. Wang, Higher order Tur{\'a}n inequalities for the partition function, Trans. Amer. Math. Soc. \textbf{372} (2019), 2143--2165.

\bibitem{dewar2013asymptotic}
M. Dewar and M. R. Murty, An asymptotic formula for the coefficients of $j(z)$, Int. J. Number Theory \textbf{9} (2013), 641--652.

\bibitem{dimitrov1998higher}
D. K. Dimitrov, Higher order Tur{\' a}n inequalities, Proc. Amer. Math. Soc. \textbf{126} (1998), 2033--2037.

\bibitem{griffin2019jensen}
M. Griffin, K. Ono, L. Rolen and D. Zagier, Jensen polynomials for the Riemann zeta function and other sequences, Proc. Natl. Acad. Sci. USA \textbf{116} (2019), 11103--11110.

\bibitem{hardy1918asymptotic}
G. H. Hardy and S. Ramanujan, Asymptotic formulae in combinatory analysis, Proc. London Math. Soc. (2) \textbf{17} (1918), 75--115.

\bibitem{jensen1913recherches}
J. L. W. V. Jensen, Recherches sur la th\'{e}orie des \'{e}quations, Acta Math. \textbf{36} (1913), 181--195.

\bibitem{li2022log}
G.-J. Li, $\ell$-Log-momotonic and Laguerre inequality of P-recursive sequences, arXiv:2206.13922, 2022.

\bibitem{liu2007log}
L. L. Liu and Y. Wang, On the log-convexity of combinatorial sequences, Adv. Appl. Math. \textbf{39} (2007), 453--476.

\bibitem{herrmann1975uber}
O. Herrmann, \"{U}ber die Berechnung der Fourierkoeffizienten der Funktion $j(\tau)$,
J. Reine Angew. Math. \textbf{274/275} (1975), 187--195.

\bibitem{hou2021log}
Q.-H. Hou and G. Li, Log-concavity of P-recursive sequences, J. Symbolic Comput. \textbf{107} (2021), 251--268.

\bibitem{hou2023log}
Q.-H. Hou and Z. Li, Log-behavior of the root sequences of P-recursive sequences, arXiv:2310.19234, 2023.

\bibitem{mahler1974coefficients}
K. Mahler, On the coefficients of transformation polynomials for the modular function,
Bull. Aust. Math. Soc. \textbf{10} (1974), 197--218.

\bibitem{liu2021inequalities}
E. Y. S. Liu and H. W. J. Zhang, Inequalities for the overpartition function, Ramanujan J. \textbf{54} (2021), 485--509.


\bibitem{petersson1932uber}
H. Petersson, \"{U}ber die Entwicklungskoeffizienten der automorphen Formen, Acta Math. \textbf{58} (1932), 169--215.

\bibitem{rademacher1938fourier}
H. Rademacher, The Fourier coefficients of the modular invariant $J(\tau)$, Am. J. Math. 
\textbf{60} (1938), 501--512.

\bibitem{sun2020log}
B. Y. Sun and J. J.-Y. Zhao, Log-behavior of two sequences related to the elliptic integrals, Acta Math. Appl. Sin. Engl. Ser. \textbf{36} (2020), 590--602.

\bibitem{wang2019higher}
L. X. W. Wang, Higher order Tur{\'a}n inequalities for combinatorial sequences, Adv. Appl. Math. \textbf{110} (2019), 180--196.

\bibitem{wang2022laguerre}
L. X. W. Wang and E. Y. Y. Yang, Laguerre inequalities for discrete sequences, Adv. Appl. Math. \textbf{139} (2022), 102357.

\bibitem{xia2018log}
E. X. W. Xia, On the log-concavity of the sequence $\{\sqrt[n]{S_{n}}\}_{n=1}^{\infty}$ for some combinatorial sequences $\{S_n\}_{n=0}^{\infty}$, Proc. Roy. Soc. Edinburgh Sect. A \textbf{148} (2018), 881--892.

\bibitem{xia2025log}
E. X. W. Xia and Z.-R. Zhang, On the log-concavity of $n$-th root of a sequence, J. Symbolic Comput. \textbf{127} (2025), 102349.

\bibitem{zhao2015log}
F.-Z. Zhao, Log-convexity of some recurrence sequences, J. Indian Math. Soc. \textbf{82} (2015), 207--224.

\bibitem{zhao2018log}
F.-Z. Zhao, The log-convexity of r-derangement numbers, Rocky Mountain J. Math. \textbf{48} (2018), 1031--1042.





\end{thebibliography}
\end{document}